\documentclass{amsart}
\usepackage[all]{xy}
\usepackage{amssymb,a4wide}

\newcommand{\IR}{\mathbb R}
\newcommand{\IN}{\mathbb N}
\newcommand{\IQ}{\mathbb Q}

\newcommand{\IZ}{\mathbb Z}
\newcommand{\I}{\mathcal I}
\newcommand{\U}{\mathcal U}

\newcommand{\IC}{\mathbb C}
\newcommand{\e}{\varepsilon}
\newcommand{\w}{\omega}
\newcommand{\C}{\mathcal C}
\newcommand{\Ra}{\Rightarrow}

\newtheorem{theorem}{Theorem}[section]
\newtheorem{problem}[theorem]{Problem}
\newtheorem{lemma}[theorem]{Lemma}
\newtheorem{corollary}[theorem]{Corollary}
\newtheorem{claim}[theorem]{Claim}
\newtheorem{proposition}[theorem]{Proposition}
\theoremstyle{definition}
\newtheorem{definition}[theorem]{Definition}
\newtheorem{remark}[theorem]{Remark}

\title{The continuity of Darboux injections between manifolds}
\author{Iryna Banakh and Taras Banakh}
\address{I.Banakh: Ya. Pidstryhach Institute for Applied Problems of Mechanics and Mathematics of NASU, Naukova 3b, Lviv}
\address{T.Banakh: Ivan Franko National University of Lviv (Ukraine) and Jan Kochanowski University in Kielce (Poland)}

\email{ibanakh@yahoo.com, t.o.banakh@gmail.com}
\begin{document}
\begin{abstract} We prove that an injective map $f:X\to Y$ between metrizable spaces $X,Y$ is continuous if for every connected subset $C\subset X$ the image $f(C)$ is connected and one of the following conditions is satisfied:
\begin{itemize}
\item $Y$ is a 1-manifold and $X$ is compact and connected;
\item $Y$ is a 2-manifold and $X$ is a closed 2-manifold;
\item $Y$ is a 3-manifold and $X$ is a rational homology 3-sphere.
\end{itemize}
This gives a partial answer to a problem of Willie Wong, posed on Mathoverflow.
\end{abstract}
\maketitle

\section{Introduction}

In \cite{Wong} Willie Wong asked the following intriguing and still open

\begin{problem}\label{prob1} Is each bijective Darboux map $f:\IR^n\to\IR^n$ a homeomorphism?
\end{problem}

We recall that a map $f:X\to Y$ between topological spaces is {\em Darboux} if for any connected subspace $C\subset X$ the image $f(C)$ is connected. Injective (bijective) Darboux functions will be called {\em Darboux injections} (resp.  {\em Darboux bijections}). A bijection of a space onto itself will be called a {\em self-bijection}.

Wong's Problem~\ref{prob1} was motivated by the following old result of Pervine and Levine \cite{PL} who generalized an earlier result of Tanaka \cite{Tan}.

\begin{theorem}[Tanaka--Pervine--Levine]\label{t:TPL} A bijective map $f:X\to Y$ between semilocally-connected Hausdorff topological spaces  is a homeomorphism if and only if both functions $f$ and $f^{-1}$ are Darboux.
\end{theorem}

A topological space $X$ is {\em semilocally-connected} if each point $x\in X$ has a base of neighborhoods $O_x\subset X$ whose complements $X\setminus O_x$ have only finitely many connected components.  Since the Euclidean spaces are semilocally-connected, Theorem~\ref{t:TPL} implies that a bijective map $f:\IR^n\to\IR^m$ is a homeomorphism if both functions $f$ and $f^{-1}$ are Darboux.

We do not know the answer to Wong's Problem~\ref{prob1} even for $n=2$, but in order to put this problem into a wider perspective, we ask a more general

\begin{problem}\label{prob2} Recognize pairs of topological spaces $X,Y$ for which every Darboux injection $f:X\to Y$ is continuous.
\end{problem}

In this paper we shall give a partial answer to Problem~\ref{prob2} for Darboux injections between manifolds of dimension $n\le 3$. 

A metrizable space $X$ is called 
\begin{itemize}
\item an {\em $n$-manifold} if every point $x\in X$ has an open neighborhood $O_x$, homeomorphic to an open subset of the half-space $\IR^n_+:=\{(x_1,\dots,x_n)\in\IR^n:x_1\ge 0\}$;
\item a {\em closed $n$-manifold} if $X$ is compact and every point $x\in X$ has an open neighborhood $O_x$, homeomorphic to $\IR^n$;
\item a ({\em closed}\/) {\em manifold} if it is a (closed) $n$-manifold for some $n\in\mathbb N$;
\item a {\em rational homology $n$-sphere} if $X$ is a closed manifold with singular homology groups $H_k(X;\IQ)\approx H_k(S^n;\IQ)$ for all $k\ge 0$.
\end{itemize}
Here $S^n:=\{x\in\IR^{n+1}:\|x\|=1\}$ stands for the $n$-dimensional sphere in the Euclidean space $\IR^{n+1}$, and $\approx$ denotes the isomorphism of groups. 

In Proposition~\ref{p:hs} we prove that a connected closed 3-manifold $X$ is a rational homology 3-sphere if and only if $H_1(X;\IQ)=0$ if and only if the homology group $H_1(X)$ is finite if and only if the fundamental group $\pi_1(X)$ has finite abelianization. According to the Poincar\'e Conjecture (proved to be true by Grigory  Perelman), each connected closed $3$-manifold with trivial fundamental group $\pi_1(X)$ is homeomorphic to the $3$-dimensional sphere. On the other hand, there are infinitely many topologically non-homeomorphic rational homology 3-spheres, see the MO-discussion at {\small\tt  https://mathoverflow.net/q/311063}.


 The following theorem is a main result of this paper.
 
\begin{theorem}\label{t:main} A Darboux injection $f:X\to Y$ between metrizable spaces is continuous if one of the following conditions is satisfied:
\begin{enumerate}
\item $Y$ is a $1$-manifold and $X$ is compact and connected;
\item $Y$ is a $2$-manifold and $X$ is a closed $2$-manifold;
\item $Y$ is a $3$-manifold and $X$ is a rational homology $3$-sphere.
\end{enumerate}
\end{theorem}

This theorem follows from Corollaries~\ref{c:main1}, \ref{c:main2}, and \ref{c:main3}, proved in Sections~\ref{s:main1}, \ref{s:main2} and \ref{s:main3}, respectively. In fact, our results hold more generally for Darboux injections into generalizations of manifolds, called varieties.  Varieties are introduced and studied in Section~\ref{s:nvar}. 

Theorem~\ref{t:main} has the following corollary, related to  Wong's Problem~\ref{prob1}.

\begin{corollary} Any Darboux self-bijection $f:S^n\to S^n$ of the sphere of dimension $n\le 3$ is a homeomorphism.
\end{corollary}

Another corollary of Theorem~\ref{t:main} establishes a dimension-preserving  property of Darboux injections between manifolds.

\begin{corollary} Let $f:X\to Y$ be a Darboux injection between manifolds. If $\dim(Y)\le 3$, then $\dim(X)\le\dim(Y)$.
\end{corollary}

\begin{proof} Assuming that $\dim(X)>\dim(Y)$, we can find subset $S\subset X$, homeomorphic to the $\dim(Y)$-dimensional sphere. Denote by $X'$ the connected component of $X$ containing $S$. Since $f$ is Darboux, the image $f(X')$ is connected and hence is contained in some connected component $Y'$ of $Y$. Applying Theorem~\ref{t:main},  we conclude that the restriction $f{\restriction}S:S\to Y'$ is continuous. By the compactness of $S$, the map $f{\restriction}S$ is a topological embedding and its image $f(S)$ is a closed subset of $Y'$. By the Invariance of Domain \cite[2B.3]{Hat}, the image $f(S)$ is open in $Y'$.  Being connected, the space $Y'$ coincides with its open-and-closed subset $f(S)$. Then $f(S)\subset f(X')\subset Y'=f(S)$ implies $X'=S$ (by the injectivity of $f$) and hence $\dim(X)=\dim(X')=\dim(S)=\dim(Y')=\dim(Y)$, which contradicts our assumption.
\end{proof}
 
We do not know if the triviality of the homology group $H_1(X;\IQ)$ is essential in Theorem~\ref{t:main}(3). Also our technique does not allow to generalize Theorem~\ref{t:main} to Darboux injections between $n$-manifolds of dimension $n>3$. So we ask two problems.  

\begin{problem} Is each Darboux self-bijection of the $3$-dimensional torus a homeomorphism?
\end{problem} 

\begin{problem} Is each Darboux self-bijection of the $4$-dimensional sphere a homeomorphism?
\end{problem} 

\section{Darboux injections into 1-manifolds}\label{s:main1}

In this section we prove Theorem~\ref{t:main}(1).  
This will be done in Corollary~\ref{c:main1} after long preliminary work made in Lemmas~\ref{l1}--\ref{l8}. 

It will be convenient to identify the circle $S^1$ and the sphere $S^2$ with the one-point compactifications $\bar\IR$ and $\bar\IC$ of the real line $\IR$ and the complex plane $\IC$, respectively.

Corollary~\ref{c:main1} will be derived from a series of lemmas, the first of which is well-known in Continuum Theory as the Boundary Bumping  Theorem \cite[5.4]{Nad}.

\begin{lemma}\label{l1} Let $X$ be a compact connected space, $U\subset X$ be a non-empty open set, not equal to $X$, and $C$ be a connected component of $U$. Then the closure $\bar C$ of $C$ in $X$ intersects $X\setminus U$.
\end{lemma}

\begin{lemma}\label{l2} Any Darboux injection $f:\IR\to\IR$ is a topological embedding.
\end{lemma}

\begin{proof} We shall show that the map $f$ is open. Indeed, for any point $x\in \IR$ and any $\e>0$ consider the open intervals $A:=(x-\e,x)$, $B:=(x,x+\e)$ and $C:=(x-\e,x+\e)=(x-\e,x)\cup\{x\}\cup(x,x+\e)$. By the Darboux property of $f$, the sets $f(A)$, $f(A)\cup\{f(x)\}$, $f(B)$ and $f(B)\cup\{f(x)\}$ are connected and non-empty. Since $f(A)$ and $f(B)$ are disjoint, the union $f(A)\cup\{f(x)\}\cup f(B)=f(C)$ is a neighborhood of the point $f(x)$ in the real line. So, the map $f$ is open and hence the map $f^{-1}:f(\IR)\to\IR$ is continuous. The set $f(\IR)$, being open and connected in $\IR$, is homeomorphic to $\IR$. Then the map $f^{-1}:f(\IR)\to\IR$, being continuous and monotone, is a homeomorphism.
\end{proof}

\begin{lemma} Any Darboux injection $f:[0,1)\to\IR$ is a topological embedding.
\end{lemma}

\begin{proof} Consider the open interval $I:=(0,1)$. By Lemma~\ref{l2}, the restriction $f{\restriction}I$ is a topological embedding. In particular, it is monotone. We lose no generality assuming that $f{\restriction}I$ is increasing. Observing that for every $x\in I$ the sets $(0,x)$ and $[0,x)$ are connected, we can prove that  $f(0)=\inf f(I)$ and $f$ is continuous at $0$. So, $f$ is a topological embedding.
\end{proof}

By analogy we can prove 

\begin{lemma}\label{l4} Any Darboux injection $f:[0,1]\to\IR$ is a topological embedding.
\end{lemma}

\begin{lemma}\label{l5} Any Darboux injection $f:X\to\IR$ of a path-connected Hausdorff space $X$ is a topological embedding.
\end{lemma}

\begin{proof} Consider  four possible cases. 
\smallskip

1. Assume that $\inf f(X)$ and $\sup f(X)$ belong to $f(X)$. In this case we can find points $a,b\in X$ such that $f(a)=\inf f(X)$ and $f(b)=\sup f(X)$. Since the space $X$ is path-connected, there exists a subspace $I\subset X$ that contains the points $a,b$ and is homeomorphic to the closed interval $[0,1]$. Then $f(I)=[f(a),f(b)]=f(X)$ and hence $X=I$. By Lemma~\ref{l4}, $f{\restriction}I$ is a topological embedding.
\vskip3pt

2. Assume that $\inf f(X)\in f(X)$ but $\sup f(X)\notin f(X)$. In this case we can choose a point $a\in X$ with $f(a)=\inf f(X)$ and a sequence of points $\{b_n\}_{n\in\w}\subset X$ such that $(f(b_n))_{n\in\w}$ is an increasing sequence of real numbers with $\lim_{n\to\infty}f(b_n)=b_\infty:=\sup f(X)$.

Since $X$ is path-connected, for any $n\in\w$ there exists a continuous injective map $\gamma_n:[0,1]\to X$ such that $\gamma_n(0)=a$ and $\gamma_n(1)=b_n$. Put $I_n:=\gamma_n([0,1])$. By Lemma~\ref{l4}, the restriction $f{\restriction}I_n:I_n\to\IR$ is a topological embedding and hence  $f(I_n)=[f(a),f(b_n)]$. The continuity of the restrictions $f^{-1}{\restriction}[f(a),f(b_n)]$ implies the continuity of the map $f^{-1}:f(X)\to X$. To see that $f$ is continuous, take any point $x\in X$ and find $n\in\w$ such that $f(b_n)>x$. The continuity of the map $f^{-1}$ implies that the set $X\setminus I_n=f^{-1}((b_n,b_\infty))$ is path-connected. Assuming that this set contains $x$ in its closure, we would conclude that $f(\{x\}\cup(X\setminus I_n))=\{f(x)\}\cup (b_n,b_\infty)$ is connected, which is not true. Consequently, $I_n$ is a neighborhood of $x$ in $X$ and the continuity of $f{\restriction}I_n$ implies the continuity of $f$ at $x$.
\smallskip

3. By analogy we can consider the case $\inf f(X)\notin f(X)$ and $\sup f(X)\in f(X)$.
\smallskip

4. Finally, we consider the case $\inf f(X)\notin f(X)$ and $\sup f(X)\notin f(X)$. In this case we can choose two sequences of points $(a_n)_{n\in\w}$ and $(b_n)_{n\in\w}$ in $X$ such that $a_0<b_0$, the sequence $(f(a_n))_{n\in\w}$ decreases to $a_\infty:=\inf f(X)$ and $(f(b_n))_{n\in\w}$ increases to $b_\infty:=\sup f(X)$. For every $n\in\w$ choose an injective continuous map $\gamma_n:[0,1]\to X$ such that $\gamma_n(0)=a_n$ and $\gamma_n(1)=b_n$. By Lemma~\ref{l4}, for the interval $I_n:=\gamma([0,1])$ the restriction $f{\restriction}I_n$ is a topological embedding. Then $f(I_n)=[a_n,b_n]$.  The continuity of the restrictions $f^{-1}{\restriction}[f(a_n),f(b_n)]$ implies the continuity of the map $f^{-1}:f(X)\to X$. To see that $f$ is continuous, take any point $x\in X$ and find $n\in\w$ such that $f(a_n)<x<f(b_n)$. The continuity of the map $f^{-1}$ implies that the sets $f^{-1}((b_n,b_\infty))$ and $f^{-1}((a_\infty,a_n))$ are path-connected. Assuming that one of these sets contains $x$ in its closure, we would conclude that one of the sets $(a_\infty,a_n)\cup\{f(x)\}$ or $\{f(x)\}\cup (b_n,b_\w)$ is connected, which is not true. Consequently, $I_n$ is a neighborhood of $x$ in $X$ and the continuity of $f{\restriction}I_n$ implies the continuity of $f$ at $x$.
\end{proof}

\begin{lemma}\label{l6} Assume that $f:X\to\IR$ is a Darboux injection of a compact connected space $X$ to the circle $\bar\IR$. Then for any point $x\in X$ the space $X\setminus\{x\}$ has at most two connected components. Moreover, for any connected component $C$ of $X\setminus\{x\}$ we have $x\in\overline{C}$ and $f(x)\in\overline{f(C)}$.
\end{lemma}

\begin{proof} Let $C$ be a connected component of $X\setminus\{x\}$.  By Lemma~\ref{l1}, $x\in\overline C$, which implies that the set $C\cup\{x\}$ is connected. Since the map $f$ is Darboux, the sets $f(C)$ and $f(C)\cup f(x)$ are connected, which implies that $f(x)\in\overline{f(C)}$.

Assuming that $X\setminus\{x\}$ contains three pairwise disjoint connected components $C_1,C_2,C_3$, we conclude that $f(C_1)$, $f(C_2)$, $f(C_3)$ are pairwise disjoint connected sets in the circle $\bar\IR$ such that $f(x_i)\in\bigcap_{i=1}^3\overline{f(C_i)}$. But the circle cannot contain such three pairwise disjoint connected sets.
\end{proof}

\begin{lemma}\label{l7} If a connected compact Hausdorff space $X$ admits a Darboux injection $f:X\to \bar\IR$ to the circle, then $X$ is locally connected.
\end{lemma}

\begin{proof} Without loss of generality, we can assume that the space $X$ contains more than one point.

Given any point $x\in X$ and open neighborhood $O_x\subset X$of $x$, it suffices to find a closed connected neighborhood $C_x\subset \bar O_x$ of $x$. Replacing $O_x$ by a smaller neighborhood, we can assume that its closure $\bar O_x$ does not contain some point $x_0\in X$.  

Fix any compact neighborhood $K_x\subset O_x$ of $x$. Let $\mathcal C$ be the family of connected components of $\bar O_x$ that intersect the compact set $K_x$. By Lemma~\ref{l1}, each set $C\in\C$ intersects the set $X\setminus O_x$. 

We claim that the family $\C$ is finite. To derive a contradiction, assume that 
$\C$ is infinite. By Lemma~\ref{l6}, the space $X\setminus\{x_0\}\supset\bar O_x$ has at most two connected components. Replacing $\C$ by a smaller infinite subfamily, we can assume that each set $C\in\C$ belongs to the same connected component $X'$ of the space $X\setminus\{x_0\}$. Since $X'$ is closed in $X\setminus\{x_0\}$ and $x_0\notin \bar O_x$, the space $X'\cap \bar O_x$ is compact.

Consider the hyperspace $\mathcal K(X'\cap \bar O_x)$ of non-empty closed subsets of $\bar O_x$, endowed with the Vietoris topology. It is well-known \cite[3.12.27]{Eng} that the space $\mathcal K(X'\cap\bar O_x)$ is compact and Hausdorff. Then the infinite set $\C$ has an accumulation point $C_\infty$ in the compact Hausdorff space $\mathcal K(X'\cap\bar O_x)$. Since each set $C\in\C$ is connected and meets the sets $K_x$ and $X\setminus O_x$, so does the set $C_\infty$. 
This implies that $C_\infty$ is a connected set containing more than two points. Since the injective map $f:X\to\bar\IR$ is Darboux, the image $f(C_\infty)$ is a connected subset of $\bar\IR\setminus\{f(x_0)\}$, containing more than one point. Then there exists a  point $c\in C_\infty$ such that the set $f(C_\infty)$ is a neighborhood of the point $f(c)$ in $\bar\IR$.

Taking into account that $f(C_\infty)\subset f(X')\subset \bar\IR\setminus\{f(x_0)\}$ is a neighborhood of $f(c)$, we  conclude that the space $f(X')\setminus \{f(c)\}$ is disconnected and so is the space $X'\setminus\{c\}$ (as $f$ is Darboux). By Lemma~\ref{l6}, the space $X'\setminus \{c\}$ has exactly two connected components $U,V$ such that $f(c)\in\overline{f(U)}\cap \overline{f(V)}$. The set $f(C_\infty)$, being a neighborhood of $f(c)$, intersects both sets $f(U)$ and $f(V)$. So, we can choose two points $u\in U\cap C_\infty$ and $v\in V\cap C_\infty$. Since $C_\infty$ is an accumulation point of the family $\C$, there exists a connected set $C\in\C$ such that $C\cap U\ne\emptyset$ and $C\cap V\ne\emptyset$. Since the point $c$ belongs to at most one set of the disjoint family $\C$, we can additionally assume that $c\notin C$. Then $C=(C\cap U)\cup (C\cap V)$ is a union of two disjoint non-empty open subsets $C\cap U$ and $C\cap V$ of $C$, which is forbidden by the connectedness of $C$.

This contradiction shows that the family $\C$ of connected components of $\bar O_x$ is finite. 
 Then $W_x=\bigcup \{C\in\C:x\in C\}$ is a connected neighborhood of $x$, contained in $\bar O_x$ and containing the neighborhood $K_x\setminus\bigcup\{C\in\C:x\notin\C\}$ of $x$. 
\end{proof}

\begin{lemma}\label{l8} Any Darboux injection $f:X\to \bar\IR$ from a connected compact metrizable space $X$ is a topological embedding.
\end{lemma}

\begin{proof} By Lemma~\ref{l7}, the space $X$ is locally connected and by the Hahn-Mazurkiewicz Theorem \cite[8.14]{Nad}, $X$ is locally path-connected and path-connected. If $f(X)\ne\bar\IR$, then $f$ is a topological embedding by Lemma~\ref{l5}. It remains to consider the case $f(X)=\bar\IR$. Since $X$ is locally path-connected, each point $x\in X$ has a compact path-connected neighborhood $K_x\subset X$, which is not equal to $X$. Then $f(K_x)\ne\bar\IR$ and we can apply Lemma~\ref{l5} to conclude that the restriction $f{\restriction}K_x$ is a topological embedding. In particular, $f$ is continuous at the point $x$. So the map $f:X\to\bar\IR$ is continuous and bijective. By the compactness of $X$, the map $f$ is a homeomorphism.
\end{proof}

Since each connected $1$-manifold embeds into the circle, Lemma~\ref{l8} implies the following corollary that proves Theorem~\ref{t:main}(1).

\begin{corollary}\label{c:main1} Any Darboux injection $f:X\to Y$ from a connected compact metrizable space to any $1$-manifold $Y$ is a topological embedding.
\end{corollary}

\section{Darboux injections into $n$-varieties}\label{s:nvar}

In fact, $n$-manifolds $Y$ in Theorem~\ref{t:main} can be replaced by their generalizations, called $n$-varieties. The definition of an $n$-variety is inductive, which will allow us to use Theorem~\ref{t:main}($n$) in the proof of Theorem~\ref{t:main}($n+1$) (for $n\in\{1,2\}$).

First we recall some notions related to separators.

\begin{definition} We say that a subset $S$ of a topological space $X$ {\em separates} a set $A\subset X$ if the complement $A\setminus S$ is disconnected. In this case $S$ is called a {\em separator} of $A$. A set $S$ is called a {\em separator} of $X$ {\em between points} $a,b\in X$ if these points belong to different connected components of $X\setminus S$.
\end{definition}

Next, we introduce a new notion of a componnectness, which is a common generalization of the notions of the compactness and the semilocally-connectedness.

\begin{definition} A topological space $X$ is called {\em componnected} if $X$ has a base $\mathcal B$ of the topology such that for any set $B\in\mathcal B$ the  complement $X\setminus B$ can be written as finite union $C_1\cup\dots\cup C_n$ of compact or connected subsets of $X$.
\end{definition}

It is clear that a topological space is componnected if it is compact or semilocally-connected. 

Now we are able to introduce the notion of an $n$-variety.

\begin{definition}\label{d:var}
A Hausdorff topological space $X$ is defined to be 
\begin{enumerate}
\item a {\em $1$-variety} if each point $x\in X$ has a neighborhood homeomorphic to $\IR$;
\item[$(n)$] an {\em $(n+1)$-variety} for $n\in\IN$ if 
\begin{itemize}
\item $Y$ is first-countable;
\item each connected component of $Y$ is componnected;
\item for any connected set $C\subset Y$ that contains more than one point and any sequence $\{y_n\}_{n\in\w}\subset Y$ that converges to a point $y\in \bar C$ there exists a compact connected $n$-variety $S\subset Y$ that separates the set $C$ and contains infinitely many points $y_n$, $n\in\w$.
\end{itemize}
\end{enumerate}
\end{definition}

\begin{remark} Each compact connected $1$-variety is homeomorphic to the circle, see, e.g., \cite{Gale}.
\end{remark} 

\begin{remark} By induction it can be proved that each $n$-manifold of dimension $n\ge 2$ is an $n$-variety. On the other hand, the Sierpi\'nski carpet also is a $2$-variety but is very far from being a $2$-manifold.
\end{remark}

In the proofs of Theorems~\ref{t:main2} and \ref{t:main3} we shall exploit the following lemma.

\begin{lemma}\label{l:main} Let $X$ be a Darboux function from a regular topological space $X$ to a topological space $Y$. If a subset $S\subset Y$ separates the image $f(X)$, then the preimage $f^{-1}(S)$ contains a closed nowhere dense separator of $X$.
\end{lemma}

\begin{proof} Since $f$ is Darboux and the space $f(X)\setminus S$ is disconnected, the subset $X\setminus f^{-1}(S)$ is disconnected, too.
So, there exist open subsets $U,V$ in $X$ such that $X\setminus f^{-1}(S)\subset U\cap V$ and the sets $U\setminus f^{-1}(S)$ and $V\setminus f^{-1}(S)$ are disjoint and non-empty. If the set $X\setminus f^{-1}(S)$ is dense in $X$, then $$\emptyset =(U\setminus f^{-1}(S))\cap(V\setminus f^{-1}(S))=U\cap V\cap(X\setminus f^{-1}(S))$$ implies $U\cap V=\emptyset$. Then the complement $L:=X\setminus(U\cup V)\subset f^{-1}(S)$ is a nowhere dense closed separator of $X$.

If the set $X\setminus f^{-1}(S)$ is not dense in $X$, then  $f^{-1}(S)$ has non-empty interior in $X$. By the regularity of $X$, this interior contains the closure $\overline U$ of some non-empty open set $U$ in $X$. Then $f^{-1}(S)$ contains the closed nowhere dense separator $L:=\bar U\setminus U$ of $X$.
\end{proof}

\section{Fences in manifolds}

A metrizable separable  space $X$ is called  a {\em fence} if $X$ is compact and  and each connected component of $X$ is homeomorphic to the segment $[a,b]\subset\IR$ for some $a\le b$. So, a singleton and the Cantor set both are fences.
By Corollary 1.9.10 \cite{End}, each fence $F$ has dimension $\dim(F)\le 1$. 

In the following lemma we shall prove that each fence $X$ has trivial \v Cech cohomology groups $\check H^n(X;G)$ for every $n\ge 1$ and every coefficient group $G$. We shall use Huber's Theorem \cite{Huber} saying that the group $\check H^n(X;G)$ is isomorphic to the group $[X,K(G,n)]$ of homotopy classes of continuous maps from $X$ to the Eilenberg-MacLane complex $K(G,n)$. By definition, $K(G,n)$ is a CW-complex that has a unique non-trivial homotopy group $\pi_n(K(G,n))$ and this group is isomorphic to $G$. By a {\em coefficient group} we understand any non-trivial countable abelian group $G$.

\begin{lemma}\label{l:Cech} Any fence $X$ has \v Cech cohomology $$\check H^n(X;G)\approx [X,K(G,n)]=0$$for any $n\ge 1$ and any coefficient group $G$.
\end{lemma}

\begin{proof} Since the Eilenberg-MacLane complex $K(G,n)$ is path-connected for $n\ge 1$, it suffices to prove that each continuous map $f:X\to K(G,n)$ is homotopic to a constant map.

Denote by $\I$ the family of connected components of the fence $X$. By definition, each connected component $I\in\I$ is homeomorphic to the segment $[a,b]\subset\IR$ for some $a\le b$.
By the Tietze-Urysohn Theorem \cite[2.1.8]{Eng}, there exists a retraction $r_I:X\to I$.

The subset $f(X)$ of the CW-complex $K(G,n)$ is compact and hence is contained in a finite subcomplex $Y$ of $K(G,n)$. By Proposition A.4 in \cite{Hat}, the finite  CW-complex $Y$ is locally contractible. Being a finite CW-complex, the space $Y$ is compact, metrizable, and finite-dimensional. By Theorem V.7.1 in \cite{Hu}, the finite-dimensional locally contractible metrizable space $Y$ is an absolute neighborhood retract. By Theorem IV.1.1 in \cite{Hu}, $Y$ has an open cover $\U$ of $Y$ such that any map $g:X\to Y$ with $(g,f)\prec\U$ is homotopic to $f$.
The notation $(g,f)\prec\U$ means that for every $x\in X$ the doubleton $\{g(x),f(x)\}$ is contained in some set $U\in\U$.

For every $z\in I$ choose an open neighborhood $W_z\subset X$ such that $f(W_z)\subset U$ for some $U\in\U$.
Next, use the continuity of the retraction $r_I$ and choose an open neighborhood $V_z\subset W_z$ of $z$ such that $r_I(V_z)\subset W_z$. 

By \cite[6.1.23]{Eng}, connected components of compact Hausdorff spaces coincide with quasicomponents. Consequently, the open neighborhood $\bigcup_{z\in I}V_z$ of $I$ contains a closed-and-open neighborhood $V_I\subset X$ of $I$. For every point $x\in V_I$ we can find a point $z\in I$ with $x\in V_z$ and conclude that $\{x,r_I(x)\}\subset W_z$ and hence $\{f(x),f\circ r_I(x)\}\subset U$ for some $U\in\U$. 

 By the compactness of $X$, the open cover $\{V_I:I\in\I\}$ of $X$ contains a finite subcover $\{V_{I_1},\dots,V_{I_n}\}$ of $X$. Define a map $r:X\to \bigcup_{k=1}^n I_k$ by the formula $r(x)=r_{I_k}(x)$ where $k$ is a unique number such that $x\in V_{I_k}\setminus\bigcup_{p<k}V_{I_p}$. It follows that $(f,f\circ r)\prec\U$ and hence the maps $f$ and $f\circ r$ are homotopic. Since the subset $J:=\bigcup_{k=1}^nI_k$ is a disjoint union of contractible sets and the space $K(G,n)$ is path-connected (for $n>0$), the map $f|J:J\to K(G,n)$ is homotopic to a constant map and so are the maps $f\circ r$ and $f$.
\end{proof}

Let $G$ be a coefficient group. A closed $n$-manifold $X$ is called {\em $G$-orientable} if its $n$-th cohomology group $\check H^n(X;G)$ is isomorphic to $G$.
It is known \cite[p.6]{DV} that each closed $n$-manifold is $\IZ_2$-orientable, where $\IZ_2$ stands for the two-element group $\IZ/2\IZ$. In the proof of the following proposition we shall use the Poincar\'e Duality Theorem saying that for a closed subset $B$ of a $G$-orientable closed $n$-manifold $X$ the \v Cech cohomology group $\check H^p(X,B;G)$ of the pair $(X,B)$ is naturally isomorphic to the singular homology group $H_{n-p}(X\setminus B;G)$ of the complement $X\setminus B$. This form of the Duality Theorem can be found in \cite[0.3.1]{DV}. For basic information on homologies and cohomologies, see  Hatcher's  monograph \cite{Hat}.

\begin{proposition}\label{p:multi} Let  $F$ be a fence in a $G$-orientable closed $n$-manifold $X$ for dimension $n\ge 3$. If for some $k\le n-2$ the singular homology group $H_k(X;G)$  is trivial, then the group $H_k(X\setminus F;G)$ is trivial, too.
\end{proposition}

\begin{proof} 
By our assumption, the homology group $H_k(X;G)$ is trivial. By the Poincar\'e Duality, $\check H^{n-k}(X,G)\cong H_k(X;G)=0$ and by Lemma~\ref{l:Cech}, the group $\check H^{n-k-1}(F;G)$ is trivial (as $n-k-1\ge 1$). Writing a piece of the long exact sequence of the pair $(X,F)$ for \v Cech cohomology:
$$0=\check H^{n-k-1}(F;G)\to \check H^{n-k}(X,F;G)\to \check H^{n-k}(X;G)=0,$$
we conclude that the group $\check H^{n-k}(X,F;G)$ is trivial. By the Duality Theorem 0.3.1 \cite{DV}, the \v Cech cohomology group $\check H^{n-k}(X,F;G)$ is isomorphic to the singular homology group $H_{k}(X\setminus F;G)$. So, the group $H_{k}(X\setminus F;G)$ is trivial.
\end{proof}

\begin{proposition}\label{p:2not} For any fence $F$ in a connected closed $n$-manifold $X$ of dimension $n\ge 2$, the complement $X\setminus F$ is connected.
\end{proposition}

\begin{proof} Since $X$ is connected, its singular homology group $H_0(X;\IZ_2)$ is isomorphic to $\IZ_2$. Since each $n$-manifold $X$ is $\IZ_2$-orientable (see \cite[p.6]{DV}), $\check H^n(X;\IZ_2)\cong \IZ_2$. By Lemma~\ref{l:Cech}, $\check H^{n-1}(F;\IZ_2)=\check H^n(F;\IZ_2)=0$. Writing a piece of the long exact sequence of the pair $(X,F)$ for \v Cech cohomology:
$$0=\check H^{n-1}(F;\IZ_2)\to \check H^n(X,F;\IZ_2)\to \check H^n(X;\IZ_2)\to\check H^n(F;G)=0,$$
we conclude that the group $\check H^{n}(X,F;\IZ_2)$ is isomorphic to $\check H^n(X;\IZ_2)\cong \IZ_2$. By the Poincar\'e Duality Theorem, the \v Cech cohomology group $\check H^{n}(X,F;\IZ_2)$ is isomorphic to the singular homology group $H_{0}(X\setminus F;\IZ_2)$. So, the group $H_{0}(X\setminus F;\IZ_2)$ is isomorphic to $\IZ_2$, which implies that the space $X\setminus F$ is connected and non-empty.
\end{proof}

\section{Darboux injections to 2-varieties}\label{s:main2}

In this section we prove Theorem~\ref{t:main}(2) and its more general version:

\begin{theorem}\label{t:main2} Let $X$ be a compact Hausdorff space that cannot be separated by a fence. Any  Darboux injection $f:X\to Y$ to a $2$-variety $Y$ is an open topological embedding.
\end{theorem}

\begin{proof} The proof of this theorem relies on the following lemma in which the injection $f$ and the spaces $X,Y$ satisfy the assumptions of Theorem~\ref{t:main2}.

\begin{lemma}\label{l:main2} For any sequence $\{y_n\}_{n\in\w}\subset Y$ that converges to a point $y\in \overline{f(X)}$ there exists a set $S\subset Y$ such that 
\begin{enumerate}
\item $S$ is homeomorphic to the circle $\bar\IR$;
\item  $f(X)\setminus S$ is disconnected;
\item the set $\{n\in\w:y_n\in S\}$ is infinite;
\item the map $f{\restriction}f^{-1}(S):f^{-1}(S)\to S$ is a homeomorphism.
\end{enumerate}
\end{lemma}

\begin{proof} By definition of a $2$-variety, there exists a set $S\subset Y$ satisfying the conditions (1)--(3). Since $f(X)\setminus S$ is disconnected, we can apply Lemma~\ref{l:main} and conclude that the set $f^{-1}(S)$ contains a closed separator $L$ of the set $X$. By our assumption, $L$ is not a fence and hence some connected component $C$ of $L$ is not homeomorphic to a subset of $[0,1]$. It follows that the restriction $f{\restriction}C:C\to S$ is a Darboux injection to the space $S$, which is homeomorphic to $\bar\IR$. By Lemma~\ref{l8}, the map $f{\restriction}C:C\to\bar\IR$ is a topological embedding. Since $C$ is not homeomorphic to a subset of $[0,1]$, $f(C)=S$ and hence $f^{-1}(S)=C$. 
\end{proof}

Now we shall derive Theorem~\ref{t:main2} from Lemma~\ref{l:main2}.

By our assumption, the space $X$ cannot be separated by a fence, which implies that $X$ is connected (otherwise it would be separated by the empty set, which is a fence).
Since the map $f$ is Darboux, the image $f(X)$ is connected. So, $f(X)$ is contained in some connected component $Y'$ of the space $Y$. By Definition~\ref{d:var}, the connected component $Y'$ of $Y$ is componnected.

Now we prove that the image $f(X)$ is closed-and-open in $Y$. Assuming that $f(X)$ is not closed in $Y$, we can choose a sequence $\{y_n\}_{n\in\w}\subset f(X)$ that converges to some point $y\in Y\setminus f(X)$. By Lemma~\ref{l:main2}, there exists a compact subset $S\subset Y$ such that $S\subset f(X)$ and $S$ contains infinitely many points $y_n$ and hence contains $\lim_{n\to\infty}y_n=y$. But this contradicts the choice of $y\notin f(X)$.

Assuming that $f(X)$ is not open in $Y$, we can choose a sequence $\{y_n\}_{n\in\w}\subset Y\setminus f(X)$ that converges to some point $y\in f(X)$. By Lemma~\ref{l:main2}, there exists a subset $S\subset Y$ such that $y_n\in S\subset f(X)$ for infinitely many numbers $n\in\w$. But this contradicts the choice of the sequence $\{y_n\}_{n\in\w}\subset Y\setminus f(X)$.

Therefore, the connected set $f(X)\subset Y'$ is closed-and-open in $Y$ and hence coincides with $Y'$ by the connectedness of $Y'$.

Now we prove that the map $f^{-1}:Y'\to X$ is continuous. To derive a contradiction, assume that $f^{-1}$ is discontinuous at some point $y\in Y$. Since the $2$-variety is first-countable, we can find a neighborhood $U\subset X$ of $f^{-1}(y)$ and a sequence $\{y_n\}_{n\in\w}\subset Y$ such that $\lim_{n\to\infty}y_n=y$ but $f^{-1}(y_n)\notin U$ for all $n\in\w$. 

By Lemma~\ref{l:main2}, there exists a subset $S\subset Y$ such that  the set $\Omega=\{n\in\w:y_n\in S\}$ is infinite and $f^{-1}{\restriction}S:S\to f^{-1}(S)$ is a homeomorphism.  By the compactness of $S$, the limit point $y$ of the sequence $\{y_n\}_{n\in\Omega}\subset S$ belongs to $S$. Then the sequence $\{f^{-1}(y_n)\}_{n\in\Omega}\subset f^{-1}(S)$ converges to $f^{-1}(y)\in U$ and hence $f^{-1}(y_n)\in U$ for all but finitely many numbers $n\in\Omega$. But this contradicts the choice of the sequence $(y_n)_{n\in\w}$.
This contradiction completes the proof of the continuity of the map $f^{-1}:Y\to X$. 

Finally, we prove that the map $f:X\to Y'\subset Y$ is continuous. To derive
 a contradiction, assume that $f$ is discontinuous at some point $x\in X$. Then we can find a neighborhood $O_{y}\subset Y$ of $y=f(x)$ whose preimage $f^{-1}(O_y)$ is not a neighborhood of $x$. Since the set $f(X)=Y'$ is open in $Y$ and the (connected) space $Y'$ is componnected, we can replace $O_y$ by a smaller neighborhood and assume that the complement $Y'\setminus O_y$ can be written as the finite union $C_1,\dots,C_n$ of compact or connected sets. The continuity of the map $f^{-1}$ ensures that the sets $f^{-1}(C_1),\dots,f^{-1}(C_n)$ are compact or connected. Since $f^{-1}(O_y)$ is not a neighborhood of $x$,
 $$x\in\overline{X\setminus f^{-1}(O_y)}=\overline{f^{-1}(Y\setminus O_y)}=\bigcup_{i=1}^n\overline{f^{-1}(C_i)}$$and hence $x\in\overline{f^{-1}(C_i)}$ for some $i$. Observe that the set $C_i$ is not compact (otherwise, $f^{-1}(C_i)$ would be compact and $x\notin f^{-1}(C_i)=\overline{f^{-1}(C_i)}$).
 
Then $C_i$ is connected and so is its preimage $f^{-1}(C_i)$ under the continuous map $f^{-1}$. Since $x\in\overline{f^{-1}(C_i)}$, the subspace $C:=\{x\}\cup f^{-1}(C_i)$ of $X$ is connected but its image $f(C)=\{f(x)\}\cup C_i$ is not (as the singleton $\{y\}=O_y\cap C_i$ is clopen in $f(C)$). But this contradicts the Darboux property of $f$. This contradiction implies that $f:X\to Y'\subset Y$ is continuous and hence an open topological embedding.
\end{proof}

Theorem~\ref{t:main2} and Proposition~\ref{p:2not} imply

\begin{corollary}\label{c:main2} Any Darboux injection $f:X\to Y$ from a  connected closed $n$-manifold $X$ of dimension $n\ge 2$ to a $2$-variety $Y$ is an open topological embedding.
\end{corollary}

Since each  $2$-manifold is a $2$-variety, Corollary~\ref{c:main2} implies the following corollary that implies Theorem~\ref{t:main2}(2).

\begin{corollary} Any Darboux injection $f:X\to Y$ from a  connected closed $n$-manifold $X$ of dimension $n\ge 2$ to a $2$-manifold $Y$ is an open topological embedding.
\end{corollary}


\section{Darboux injections into 3-varieties}\label{s:main3}

We recall that a {\em Peano continuum} is a connected locally connected compact metrizable space. By the Hahn-Mazurkiewicz Theorem \cite[8.14]{Nad}, each Peano continuuum is a continuous image of the unit interval $[0,1]$.

\begin{theorem}\label{t:main3} Let $X$ be a Peano continuum such that 
for every fence $K\subset X$ the space $X\setminus K$ has trivial first homology group $H_1(X\setminus K;G)$ for some coefficient group $G$. Any Darboux injection $f:X\to Y$ to a $3$-variety is an open topological embedding.
\end{theorem}

By analogy with the proof of Theorem~\ref{t:main2}, Theorem~\ref{t:main3} can be derived from the following lemma in which the function $f:X\to Y$ and spaces $X,Y$ satisfy the assumptions of Theorem~\ref{t:main3}. 

\begin{lemma}\label{l:main3} For any sequence $\{y_n\}_{n\in\w}\subset Y$ that converges to a point $y\in \overline{f(X)}$ there exists a set $S\subset Y$ such that 
\begin{enumerate}
\item $S$ is a compact connected $2$-variety;
\item  $f(X)\setminus S$ is disconnected;
\item the set $\{n\in\w:y_n\in S\}$ is infinite;
\item the map $f{\restriction}f^{-1}(S):f^{-1}(S)\to S$ is a homeomorphism.
\end{enumerate}
\end{lemma}

\begin{proof}  By the definition of a $3$-variety, there exists a set $S\subset Y$ satisfying the conditions (1)--(3). By Lemma~\ref{l:main}, the preimage $f^{-1}(S)$ contains a closed separator $C$ of the set $X$. Fix any distinct points $a,b\in X$ that belong to different connected components of $X\setminus C$. So, $C$ is a separator of $X$ between the points $a$ and $b$. A closed separator $M$ of $X$ between $a,b$ is called an {\em irreducible separator between} $a,b$ if $M$ coincides with each closed separator $F\subset M$ of $X$ between $a$ and $b$.

\begin{claim} The separator $C$ of $X$ contains a closed irreducible separator $M\subset C$ between the points $a$ and $b$.
\end{claim}

\begin{proof}  Let $\mathcal S$ be the family of closed subsets $D$ of $C$ that separate $X$ between the points $a$ and $b$. The family $\mathcal S$ is partially ordered by the inclusion relation.

Let us show that for any linearly ordered subfamily $\mathcal L\subset\mathcal S$ the intersection $\bigcap\mathcal L$ belongs to $\mathcal S$. First observe that $\bigcap\mathcal L$ is a closed subset of $C$, being the intersection of closed sets in $L$.

Assuming that $\bigcap\mathcal L\notin \mathcal S$, we conclude that the points $a,b$ belong to the same connected component of $X\setminus\bigcap\mathcal L$. 
Since the space $X$ is locally path-connected, so is its open subset $X\setminus\bigcap\mathcal L$. Then we can find a compact connected subset $P\subset X\setminus\bigcap\mathcal L$ containing the points $a,b$. Taking into account that  for every $L\in\mathcal L$ the points $a,b$ lie in different connected components of $X\setminus L$, we conclude that the intersection $L\cap P$ is not empty. By the compactness of $P$ the linearly ordered family $\{L\cap P:L\in\mathcal L\}$ of non-empty closed subsets of $P$ has non-empty intersection, which coincides with the intersection $P\cap\bigcap\mathcal L=\emptyset$. The obtained contradiction completes the proof of the inclusion $\bigcap\mathcal L\in\mathcal L$. 

By the Zorn Lemma, the family $\mathcal S$ contains a minimal element $M\in\mathcal  S$. It is clear that $M$ is irreducible separator of $X$ between the points $a,b$. 
\end{proof}

\begin{claim}\label{cl:main3} The space $M$ cannot be separated by a fence.
\end{claim}

\begin{proof} To derive a contradiction, assume that $M$ is separated by some fence $J\subset M$.

Then $M\setminus J=U_1\cup U_2$ for some disjojnt non-empty open subspaces $U_1,U_2$ of $M$. Consider the closed subsets $M_1:=M\setminus U_1$ and $M_2:=M\setminus U_2$ in $M$ and observe that $M_1\cap M_2=J$ and $M_1\cup M_2=M$. 

Fix any non-zero element $g\in G$ and consider the singular $0$-cycle $c:=g\cdot(a-b)$ in $X$. 
The minimality of $M$ ensures that for every $i\in\{1,2\}$ the points $a,b$ belong to the same connected component of the locally path-connected space $X\setminus M_i$. Consequently, the points $a,b$ can be linked by a path in $X\setminus M_i$, which implies that the cycle $c$ represents zero in the homology groups $H_1(X\setminus M_1;G)$ and $H_1(X\setminus M_2;G)$. By our assumption, $H_1(X\setminus J;G)=0$.

Writing the initial piece of the Mayer-Vietoris exact sequence of the pair $(X\setminus M_1,X\setminus M_2)$, we obtain the exact sequence of Abelian groups:
$$
\xymatrix{
0=H_1(X\setminus J;G)\ar^{\partial_*}[r]& H_0(X\setminus M;G)\ar^<<<<{(e_1,e_2)}[r]&H_0(X\setminus M_1;G)\oplus H_0(X\setminus M_2;G),
}
$$
where $e_i:H_0(X\setminus M;G)\to H_0(X\setminus M_i;G)$ is the homomorphism induced by the identity inclusion $X\setminus M\hookrightarrow X\setminus M_i$ for $i\in\{1,2\}$.

 Since the 0-cycle $c$ represents zero in the homology groups $H_0(X\setminus M_1;G)$ and $H_0(X\setminus M_2;G)$, its homology class $[c]\in H_0(X\setminus M;G)$ is annulated by the homomorphism $(e_1,e_2)$. By the exactness of the Mayer-Vietoris sequence, $[c]=0$ in $H_0(X\setminus M;G)$, which is not true as $a$ and $b$ are contained in different connected components of $X\setminus M$.
\end{proof}

By Claim~\ref{cl:main3}, the space $M$ cannot be separated by a fence. Since $S$ is a  $2$-variety, we can apply Theorem~\ref{t:main2} and conclude that $f{\restriction}M:M\to S$ is an open topological embedding. The compactness of $M$ ensures that $f(M)$ is closed-and-open subset in $S$ and hence $f(M)=S$ by the connectedness of $S$.
Since $f$ is injective, $M=f^{-1}(S)$ and $f{\restriction}f^{-1}(S):f^{-1}(S)\to S$ is a homeomorphism. This completes the proofs of Lemma~\ref{l:main3} and Theorem~\ref{t:main3}.
 \end{proof}








The following proposition characterizes rational homology 3-spheres.

\begin{proposition}\label{p:hs} For a connected closed $3$-manifold $X$ the following conditions are equivalent:
\begin{enumerate}
\item $X$ is a rational homology $3$-sphere;
\item the first rational homology group $H_1(X;\IQ)$ is trivial;
\item the first integral homology group $H_1(X)$ is finite;
\item the abelianization of the fundamental group $\pi_1(X)$ is finite.
\end{enumerate}
\end{proposition} 

\begin{proof} The implication $(1)\Ra(2)$ is trivial.
\smallskip

$(2)\Ra(1)$ Assume that $H_1(X;\IQ)=0$. Since the 3-manifold $X$ is connected, it is path-connected and hence $H_0(X;\IQ)\approx \IQ\approx H_0(S^3;\IQ)$.  By Theorem~3.26 in \cite{Hat}, $\dim H_3(X;\IQ)\le 1$ and by Lemma 2.34 and Corollary 3A.6 in \cite{Hat}, $H_k(X;\IQ)=0=H_k(S^3;\IQ)$ for all $k>\dim(X)=3$.

By Corollary 3.37 in \cite{Hat}, closed manifolds of odd dimension have zero Euler characteristic. Consequently, $\chi(X)=0$. By Theorem 2.44 \cite{Hat}, 
\begin{multline*}0=\chi(X)=\dim H_0(X;\IQ)-\dim H_1(X;\IQ)+\dim H_2(X;\IQ)-\dim H_3(X;\IQ)=\\
=1-0+\dim H_2(X;\IQ)-\dim H_3(X;\IQ),
\end{multline*}
which implies that $\dim H_2(X;\IQ)=0$ and $\dim H_3(X;\IQ)=1$. Then $H_2(X;\IQ)=0=H_2(S^3;\IQ)$ and $H_3(X;\IQ)\approx \IQ\approx H_3(S^3;\IQ)$. Now we see that $H_k(X;\IQ)\approx H_k(S^3;\IQ)$ for all $k\ge 0$, which means that $X$ is a rational homological $3$-sphere.
\smallskip

$(2)\Ra(3)$ Assume that $H_1(X;\IQ)=0$. By Corollaries A.8 and A.9 in \cite{Hat}, the homology group $H_1(X)$ of the compact 3-manifold $X$ is finitely generated. By Corollary 3A.6 in \cite{Hat}, $H_1(X)\otimes\IQ\approx H_1(X;\IQ)=0$, which implies that the group $H_1(X)$ is a torsion group and being finitely generated is finite.
\smallskip

$(3)\Ra(2)$ Assuming that the homology group $H_1(X)$ is finite and applying Corollary 3A.6 in \cite{Hat}, we conclude that $H_1(X;\IQ)\approx H_1(X)\otimes \IQ=0$. 
\smallskip

The equivalence $(3)\Leftrightarrow(4)$ follows from Theorem 2A.1 \cite{Hat} saying that the homology group $H_1(X)$ of the path-connected space $X$ 
is isomorphic to the abelianization of the fundamental group $\pi_1(X)$.
\end{proof}

\begin{corollary}\label{c:main3}  Any Darboux injection $f:X\to Y$ from a rational homology $3$-sphere $X$ to a $3$-variety $Y$ is an open topological embedding.
\end{corollary}

\begin{proof} Since $H_3(X;\IQ)\approx H_3(S^3;\IQ)\approx \IQ$, the closed 3-manifold $X$ is $\IQ$-orientable (see Theorem 3.26 in \cite{Hat}). By our assumption, $X$ is $\IQ$-simply-connected. So, $X$ is path-connected and has trivial homology group $H_1(X;\IQ)$. By Proposition~\ref{p:multi}, for any fence $A\subset X$ the complement $X\setminus A$ also has trivial homology group $H_1(X\setminus A;\IQ)$. Applying Theorem~\ref{t:main3}, we conclude that each Darboux injection $f:X\to Y$ is an open topological embedding.
\end{proof}

Since each $3$-manifold is a $3$-variety, Corollary~\ref{c:main3} implies the following corollary that proves Theorem~\ref{t:main}(3).

\begin{corollary} Any Darboux injection $f:X\to Y$ from a rational homology $3$-sphere $X$ to a $3$-manifold $Y$ is an open topological embedding.
\end{corollary}

\section{ Acknowledgement}

The authors express their sincere thanks to Du\v san Repov\v s, Chris Gerig and Yves de Cornulier for their help in understanding complicated techniques of Algebraic Topology that were eventually used in some proofs presented in this paper. Also special thanks are due to Alex Ravsky for a careful reading the final version of the paper and many valuable remarks.

\end{document}